\documentclass[11pt]{amsart}
\usepackage{amsmath,amsfonts,amssymb,amscd,amsthm,amsbsy}
\numberwithin{equation}{section}
\newtheorem{thm}{Theorem}[section]

\newtheorem*{Thm}{Theorem}
\newtheorem{lem}[thm]{Lemma}

\newtheorem*{Cor}{Corollary}
\newtheorem{prop}[thm]{Proposition}
\newtheorem{defn}[thm]{Definition}

\def\bone{\mathop{{\bf1}\mkern-9.5mu \raise.4ex
	\hbox{$\scriptscriptstyle|$}\,}\nolimits}
\def\ds{\displaystyle}
\def\nat{{\mathbb N}}
\def\real{{\mathbb R}}

\def\I{\operatorname{I}}
\def\II{\operatorname{I\!I}}
\def\III{\operatorname{I\!I\!I}}

\def\IV{\operatorname{I\!V}}

\def\Id{\operatorname{Id}}

\def\sign{\operatorname{sign}}
\def\tr{\operatorname{tr}}
\def\liminf{\mathop{\rm lim\, inf\vphantom{p}}\nolimits}
\begin{document}
\title{Convergence of nonlocal threshold dynamics approximations to 
front propagation}
\author{Luis A. Caffarelli and Panagiotis E. Souganidis} 
\thanks{Department of Mathematics, University of Texas at Austin,
1 University Station C1200, Austin, TX 78712-0257 USA.\quad
email: caffarel@math.utexas.edu; 
souganid@math.utexas.edu\ \ }
\thanks{Both authors partially supported by the National Science Foundation.}
\begin{abstract}
In this note we prove that appropriately scaled threshold dynamics-type 
algorithms corresponding to the 
fractional Laplacian of order $\alpha \in (0,2)$ converge to moving fronts. 
When $\alpha \geqq 1$ the resulting interface moves by weighted 
mean curvature, while for $\alpha <1$ the normal velocity is nonlocal 
of ``fractional-type.''
The results easily extend to general nonlocal anisotropic 
threshold dynamics schemes. 
\end{abstract}

\maketitle      
\markboth{Caffarelli and Souganidis}
{Convergence of nonlocal threshold dynamics approximations to 
front propagation}

\abovedisplayskip=9pt plus 3pt minus 1pt
\belowdisplayskip=9pt plus 3pt minus 1pt
\abovedisplayshortskip=6pt plus 3pt
\belowdisplayshortskip=6pt plus 3pt

\section*{Introduction} 

We study here the convergence of a class of threshold dynamics-type 
approximations to moving fronts.
Although the arguments extend easily to general anisotropic kernels 
to keep the presentation simple, here we concentrate on a 
particular isotropic case, namely, fractional Laplacian of order 
$\alpha \in (0,2)$.
The resulting interfaces move either 
by weighted mean curvature, if $\alpha\in[1,2)$, 
or by a nonlocal fractional-type normal velocity, if $\alpha\in(0,1)$.

Threshold dynamics is a general term used to describe approximations to 
motion of boundaries of open sets in $\real^N$ by ``measuring'' interactions
with the environment. 
The general scheme we consider here is described as follows:

Let $\Omega_0$ be an open subset of $\real^N$ with boundary $\Gamma_0$.
The goal is to come up with an explicit approximation evolution scheme
with time step $h>0$, so that, as $nh\to t$, the ``approximate'' front 
$\Gamma_{nh}^h$, which is the boundary of an open set $\Omega_{nh}$ 
identified as the level set of a sign-function, 
 converges, in a suitable sense, to a moving front $\Gamma_t$, 
the boundary of an open set $\Omega_t$, 
and to identify the limiting velocity.
For each $n\in\nat$, let 
$$\Gamma_{nh}^h = \partial \{ x\in \real^N : u_h (\cdot, nh) =1\}\ ,$$
where 
\begin{equation}\label{eq0.0}
u_h (\cdot,0) = \bone_{\Omega_0}  - 
\bone_{\bar\Omega_0^c} \quad \text{in}\quad \real^N\ , 
\end{equation}
and, for $n\geqq 1$, 
\begin{equation}\label{eq0.01}
u_h (\cdot,(n+1)h) = \text{sign }(J_h * u_h (\cdot,nh))\ .
\end{equation}

Here 
$\text{sign}(t) = 1$ if $t>0$ and $-1$ otherwise, 
$\bone_A$ denotes the characteristic function of $A\subset \real^N$, and,  
for $\alpha \in(0,2)$,
\begin{equation}\label{eq0.2}
J_h (x) = p_\alpha (x, \sigma_\alpha (h)) \ ,
\end{equation}
where $p_\alpha$   
is the fundamental solution of the fractional Laplacian pde
\begin{equation}\label{eq0.3}
W_t - L^\alpha W=0 \quad\text{in}\quad \real^N\times (0,\infty)\ ,
\end{equation}
with 
\begin{equation}\label{eq0.4}
L^\alpha W(x) = \int |y-x|^{-(N+\alpha)} (W(y) -W(x))
\,dy\ .
\end{equation}

Hence, at each time step, we solve the equation~\eqref{eq0.3} with initial 
datum
$$W(\cdot,0) = u_h (\cdot,nh) \ \text{ in }\ \real^N\ ,$$
and for time $\sigma_\alpha (h)$.
Then we define $u_h (\cdot,(n+1)h)$ by 
$$u_h(x,(n+1)h) = 
\begin{cases}
1&\text{if\ \ $W(x,\sigma_\alpha (h)) >0$},\\
-1&\text{otherwise.}
\end{cases}$$

The algorithm generates functions $u_h (\cdot,nh)$  and 
open sets $\Omega_{nh}^h$ defined by 
$$\Omega_{nh}^h = \{x\in\real^N :J_h * u_h (\cdot,(n-1)h)(x)>0\}\quad
\text{and}\quad 
u_h (\cdot,nh) = \bone_{\Omega_{nh}^h}
-\bone_{(\Omega_{nh}^h)^c} \ \text{ in }\ \real^N\ .$$

We prove that, when $h\to0$, the discrete evolution $\Gamma_0 \to \Gamma_{nh}^h
= \partial \Omega_{nh}^h$ converges, in a suitable sense, to the motion 
$\Gamma_0 \to \Gamma_t$ with nonlocal fractional 
normal velocity, if $\alpha \in(0,1)$,  
and normal velocity equal to a multiple of the mean curvature, if 
$\alpha \in[1,2)$. 

If the kernel $J$ is a Gaussian, it is a classical result that the 
algorithm generates movement by mean curvature --- see below for an 
extensive discussion and references.

To state the main result we recall that the geometric evolution of a front 
$\Gamma_t=\partial \Omega_t$ with normal velocity
$v(Dn,n,\Omega_t)$, starting at $\Gamma_0 = \partial\Omega_0$,  
is best described by ``the level set'' partial differential equation  
\begin{equation}\label{pde}		
\begin{cases} \ds 
u_t + F(D^2 u, Du, \{u(\cdot,t)\geqq u(x,t)\}, \{u(\cdot,t)\leqq u(x,t)\} 
=0\quad\text{in}\quad
\real^N\times (0,\infty)\ ,\\
\noalign{\vskip6pt}
u=g\quad\text{on}\quad \real^N\times \{0\}\ ,
\end{cases}
\end{equation}
with $g$ such that 
\begin{equation}\label{pde1} 
\Omega_0 = \{x\in \real^N :g (x) >0\}\quad\text{and}\quad  
\Gamma_0 = \{x\in \real^N :g (x) =0\}\ ,
\end{equation}
and 
\begin{equation*}
\begin{split}
&F(D^2 u, Du, \{u(\cdot,t)\geqq u(x,t)\},\{u(\cdot,t)\leqq u(x,t)\} ) \\
\noalign{\vskip6pt}
&\qquad = |Du| v( -D(\frac{Du}{|Du|}) , - \frac{Du}{|Du|}, 
\{u(\cdot,t)\geqq u(x,t)\}, \{u(\cdot,t) < u(x,t)\} )\ .
\end{split}
\end{equation*}


The basic fact of the level set approach 
is that the sets $\Omega_t = \{x\in\real^N :
u(x,t) >0\}$ and $\Gamma_t = \{x\in\real^N :u(x,t)=0\}$ are independent 
of the choice of the initial datum $g$ provided the latter is positive 
in $\Omega_0$ and zero in $\Gamma_0$. 


The weighted mean curvature motion corresponds to the level set pde 
\begin{equation}\label{motion-by-mean}
u_t - C_\alpha \tr (I-\widehat{Du}\otimes \widehat{Du}) D^2 u =0
\ \text{ in }\ \real^N\times (0,\infty)\ ,
\end{equation}
where, for $p\in\real^N\setminus \{0\}$, $\hat p = p/|p|$, while the equation 
corresponding to the nonlocal motion is 
\begin{equation}\label{nonlocal-motion}
u_t - C_\alpha |Du| \int (\bone^+ (u(x+y,t) - u(x,t)) 
- \bone^- (u(x+y,t) - u(x,t))) |y|^{-(N+\alpha)}\,dy =0\ , 
\end{equation}
where $\bone^+$ and $\bone^-$ denote respectively the characteristic 
functions of $[0,\infty)$ and $(-\infty,0)$ and,  
in both cases,  $C_\alpha$ is an explicit constant specified later in the paper.

Although the heuristic meaning of \eqref{motion-by-mean} is well known, 
some discussion about \eqref{nonlocal-motion} is in order 

It is implicit in \eqref{nonlocal-motion} that there are sufficient 
cancellations in the term 
$\bone^+ (u(x+y,t)- u(x,t)) - \bone^- (u(x+y,t)- u(x,t))$ 
to compensate for the lack of integrability of the kernel 
$y\mapsto |y|^{-N-\alpha}$ at the origin. 
Indeed 
if we write the integral in polar coordinates, 
\begin{equation*}
\int_0^\infty r^{-(1+\alpha)} \int_{S^1} \bone^+ (u(x+r\sigma) 
- \bone^- (u(x+r\sigma)\, d\sigma\ ,
\end{equation*}
we see that 
the spherical integral measures the ``deviation'' of $\partial\Omega$,
where $\Omega = \{y\in\real^N :u(y+x,t) \geqq u(x,t)\}$, to be 
perfectly balanced between positive and negative parts, a deviation that 
infinitesimally is given by the mean curvature of $\partial\Omega$. 
Observe also that if the surface 
$\{ x\in\real^N : u(y+x,t) = u(x,t)\}$ is smooth and $|Du|\ne0$ on it, 
then an integration by parts leads to 
$$u_t + \frac{C_\alpha}{\alpha} |Du| \int_{\{y\in\real^N : u(x+y,t)=u(x,t)\}}
 \frac{y}{|y|^{N+\alpha}}\cdot 
\frac{Du(x+y,t)}{|Du(x+y,t)|} \,d\Sigma (y)\ ,$$
where $d\Sigma$ is the $(N-1)$-surface measure. 

To state the main result of the paper we recall that,  
given a bounded sequence $(u_h(\cdot,nh))_{n\in\nat}$ of bounded 
functions, the ``half relaxed'' limits $u^*$ and $u_*$ 
are defined by 
\begin{equation}\label{half-relaxed} 
\begin{cases}
u^* (x,t) = \limsup\nolimits^* u_h (x,t) 
= \limsup\limits_{y\to x\,,\ nh\to t} \mkern-8mu u_h (y,nh)\ ,\\
\qquad \text{and}\\
\noalign{\vskip6pt}
u_*(x,t) = \liminf_* u_h (x,t) 
= \!\! \liminf\limits_{y\to x\, ,\ nh\to t} \mkern-12mu u_h (y,nh)\ .
\end{cases}
\end{equation}

It is immediate that $u_* \leqq u^*$ and, more importantly, 
if $\liminf_* u_h= \limsup^* u_h$, then, as $h\to0$,  
$u_h\to u$ locally uniformly.

The result is:

\begin{Thm}
Assume that $\Gamma_0 = \partial \Omega_0 = \partial (\real^N\setminus 
\overline{\Omega}_0)$ and consider the family $(u_h(\cdot,nh))_{n\in\nat}$
defined by \eqref{eq0.0} 
and \eqref{eq0.01} with 
$\sigma_\alpha (h) = h^{\alpha/2}$, if $\alpha \in (1,2)$, 
$h= \sigma_1^2 (h) |\ln\sigma_1(h)|$ if $a=1$, and 
$\sigma_\alpha (h) =h^{\frac{\alpha}{1+\alpha}}$, if $\alpha \in (0,1)$. 
Let $\Omega_t = \{x\in\real^N :u(x,t) >0\}$ and 
$\Gamma_t = \{ x\in\real^N :u(x,t) =0\}$, where, for some uniformly 
continuous $g$ such that $\Omega_0 =\{x\in \real^N :g(x)>0\}$ and 
$\Gamma_0 = \{ x\in \real^N :g(x)=0\}$, $u$ is the unique uniformly 
solution of \eqref{pde} with $F$ given by \eqref{motion-by-mean}, 
if $\alpha \geqq 1$, and \eqref{nonlocal-motion}, if $\alpha <1$.
Then 
$$\liminf\nolimits_* u_h =1\ \text{ in }\ \Omega_t\quad\text{and}\quad 
\limsup\nolimits^* u_h =-1\ \text{ in }\ (\Omega_t \cup \Gamma_t)^c\ .$$
\end{Thm}

The Theorem asserts that the scheme characterizes the evolution of the 
front by assigning the values~1 inside the region $\{ u>0\}$ and $-1$  
outside the region $\{u<0\}$. 
Whether the regions where $u_h$ converges to 1 and $-1$ are exactly the 
regions inside and outside the front respectively depends on whether 
the front develops an interior or not. 
Interior are regions (patches) of positive measure where $u=0$.
Actually the answer is yes if and only if no interior develops.

We have:

\begin{Cor}		
If $\cup_{t>0}\Gamma_t \times \{t\} = \partial \{ (x,t) : u(x,t) >0\} 
= \partial \{ (x,t) : u(x,t)<0\}$, 
then the set $F^n = \cup_{n\in\nat} (\Gamma_{nh}^h \times \{nh\})$ converges, 
as $n\to\infty$, to $F =\cup_{t>0} (\Gamma_t \times \{t\})$ in the Hausdorff 
distance. 
\end{Cor}

The strategy of the  
proof of the Theorem is similar to the one in  Barles and Georgelin 
\cite{BG}, which is based on the general scheme developed by Barles and 
Souganidis \cite{BS1} for convergence of monotone, stable and consistent
approximation to viscosity solutions.
Once the correct scaling $\sigma_\alpha (h)$ is identified, the main step is 
to prove the consistency of the scheme.     
The key difference between previous results and the one here is 
that all previous works considered kernels that were either 
exponentials or compactly supported, while in the case at hand they only 
have a prescribed power decay. 

Threshold dynamics schemes are used in probability and, in 
particular, percolation theory to find asymptotic shapes.
We refer to Gravner and Griffeath  \cite{GG} and the references 
therein for a discussion and results from the probabilistic point of view.

In the context of moving fronts, Bence, Merriman and Osher \cite{MBO} 
introduced a scheme to compute mean curvature motion by iterating the 
heat equation. 
Evans \cite{E} and Barles and Georgelin \cite{BG} provided proofs for the 
BMO-algorithm. 
Some extensions to other isotropic kernels were considered by 
Ishii \cite{Is}, while Ishii, Peres and Souganidis \cite{IPS} studied general 
anisotropic schemes with compactly supported kernels, and   
Slepcev \cite{S} proved the convergence of a class of nonlocal 
threshold dynamics. 
Recently, Da Lio, Forcadel, 
and Monneau \cite{DFM} studied the convergence, at large scales, of a 
nonlocal first-order equation to an anisotropic mean curvature motion. 
Although related to, the results of \cite{DFM} are different than ours.
Nonlocal operators and, in particular, fractional Laplacians of order 
$\alpha=1$ are often used to describe dislocation dynamics by line tension 
terms deriving from an energy associated to the dislocation line. 
We refer to Garroni and M\"uller \cite{GM1}, \cite{GM2} for a variational 
analogue of what we are doing here for $\alpha=1$.
For the case $\alpha <1/2$, the stationary solutions  of \eqref{nonlocal-motion}
must satisfy an integral ``zero mean curvature'' equation, that can be 
obtained as the Euler-Lagrange equation of a minimization process in the 
Hilbert space $H^{\alpha/2}$ of functions with ``$\alpha/2$ derivatives in 
$L^2$.''
A regularity theory of such surfaces, 
similar to the classical theory of 
``boundaries of sets with minimal perimeter,''
is being developed by Caffarelli, Roquejoffre and Savin \cite{CRS}. 
Finally, Imbert and Souganidis \cite{ImS} studied recently the onset of 
fronts at the asymptotic limit of fractional integral equations with 
reaction terms. 

Nonlocal phase transistion, models were proposed by Chen and Fife \cite{CF}, 
Giacomini and Lebowitz \cite{GL1}, \cite{GL2}, 
DeMasi, Orlandi, Presutti and Triolo \cite{DOPT}, 
and DeMasi, Gobron and Presutti \cite{DGP}, 
in the Landau-Ginzburg context of mean field theory for statistical mechanics.
All the above references assume, however, 
fast enough decay or compact support for the diffusion kernels 
to guarantee an infinitesimal curvature condition for the limit.
The connection between these nonlocal equations and the underlying 
stochastic Ising systems and moving fronts was established by 
Katsoulakis and Souganidis \cite{KS1,KS2} and 
Barles and Souganidis \cite{BS2}. 

The paper is organized as follows: 
In Section~1 we present some preliminaries.  
The proof of the Theorem begins in Section~2. 
The key  step in the proof, i.e., the consistency of the scheme, 
is presented in Section~3.

\section{Preliminaries}

We recall some basic facts from the Crandall-Lions \cite{CL}  
theory of viscosity solutions that we will be using in the paper. 
We begin with  the definitions.
Since the two equations \eqref{motion-by-mean} and 
\eqref{nonlocal-motion} are of different nature, i.e., local versus 
nonlocal, we give two separate definitions.

%

\begin{defn}\label{defn:usc}
An upper semicontinuous (usc) (resp. lower semicontinuous 
(lsc)) function $u$ is a viscosity subsolution (resp. supersolution) of 
\eqref{motion-by-mean}
if and only if, for all $\phi\in C^2 (\real^N\times (0,\infty))$, 
if, for any maximum (resp.  minimum) point 
$(x,t)\in\real^N\times (0,+\infty)$ 
of $u-\phi$, 
\begin{equation}\label{eq:usc1}
\phi_t \leqq C_\alpha \tr ( I- \widehat{D\phi}\otimes \widehat{D\phi}) D^2\phi 
\quad\text{if}\quad |D\phi|\ne0
\quad\text{or}\quad 
\phi_t \leqq 0\quad\text{if}\quad |D\phi| =0
\ \text{ and }\ D^2\phi=0\ ,
\end{equation}
(resp.
\begin{equation}\label{eq:usc2}
\phi_t \geqq C_\alpha \tr ( I- \widehat{D\phi}\otimes \widehat{D\phi}) D^2\phi
\quad\text{if}\quad |D\phi|\ne0
\quad\text{or}\quad 
\phi_t \geqq 0\quad\text{if}\quad |D\phi| =0 )
\ \text{ and }\ D^2\phi=0\ .
\end{equation}
\end{defn}

To make precise statements for the nonlocal motion, it is necessary to 
introduce some additional notation. 
To this end, for $v:\real^N\times (0,\infty) \to \real$ and
$A\subset \real^N$, let 
\begin{equation}\label{nl1}
\left\{\begin{array}{l}
\ds \bar I [v](x,t) 
= \int \bone^+ (v(y +x,t) -v(x,t)) - \bone^- (v(y+x,t) - v(x,t)) 
|y|^{-(N+\alpha)}\,dy\ ,\\
\noalign{\vskip6pt}
\ds \underline I [v](x,t) 
= \int 
\bone_+ ( v(y+x,t) - v(x,t))  - \bone_- ( v(y+x,t) -v(x,t)) 
 |y|^{-(N+\alpha)}\, dy\ ,\\
\noalign{\vskip6pt}
\ds \bar \I_A [v] (x,t) 
= \int_A 
\bone^+ (v(y+x,t) - v(x,t)) - \bone^- (v(y+x,t) - v(x,t))
|y|^{-(N+\alpha)}\, dy\ ,\\
\noalign{\vskip6pt}
\ds \underline \I_A [v](x,t) 
= \int_A 
\bone_+ (v(y+x,t) - v(x,t)) - \bone_- (v(y+x,t) - v(x,t)) 
|y|^{-(N+\alpha)}\, dy\ ,\\
\end{array}\right.
\end{equation}
where $\bone_+$ and $\bone_-$ denote respectively the characteristic 
functions of $(0,\infty)$ and $(-\infty,0]$. 

We rewrite the level set pde \eqref{nonlocal-motion} 
obtained for $\alpha <1$ as 
\begin{equation}\label{lset2}
u_t - C_\alpha |Du| \bar I [u] = 0\ \text{ in }\ \real^N\times (0,\infty)\ .
\end{equation}

We have:

\begin{defn}
A locally bounded usc (resp. lsc) function $u$ is a viscosity subsolution 
(resp. supersolution) of \eqref{lset2} if and only if, for all $\phi \in 
C^2 (\real^N \times (0,\infty))$, if 
for any maximum (resp. minimum) point 
$(x,t) \in \real^N \times (0,\infty)$ 
of $u-\phi$,  and 
for any ball $B_\delta \subset \real^N$ centered at $(x,t)$,
\begin{equation}\label{usc3}
\phi_t \leqq C_\alpha |D\phi| \Big[ \bar \I_{B_\delta} [\phi] 
+ \bar \I_{\real^N\setminus B_\delta} [u]\Big]
\end{equation}
(resp. 
\begin{equation}\label{usc4}
\phi_t \geqq C_\alpha |D\phi| \Big[ \underline \I_{B_\delta} [\phi] 
+ \bar \I_{\real^N\setminus B_\delta} [u]\Big]\ )\ .
\end{equation}
\end{defn}

Few comments are in order here.
Firstly, the difference between \eqref{usc3} and \eqref{usc4} is not a typo.
It is actually necessary to guarantee the well posedness and, in particular, 
the stability of the solution --- see \cite{S} for a discussion of a 
similar problem.
Secondly, it turns out (see Barles and Imbert \cite{BI}) that 
Definition~1.2 is actually independent of $\delta$. 
Hence, we may assume in the proofs that $\delta$ is either fixed depending 
on $\phi$, or, even, that $\delta \to0$ in an appropriate way. 

It is well known that the initial value problem \eqref{motion-by-mean} has 
a unique uniformly continuous solution --- see, for example, 
Barles, Soner and Souganidis \cite{BSS} 
and Ishii and Souganidis \cite{IsS} for general results.
The well-posedness of uniformly continuous solutions of the initial value 
problem for \eqref{nonlocal-motion}, which
follows along the lines of the analogous 
result for integro-differential operators, 
has been studied recently by Imbert \cite{Im}. 

It turns out, however, (see \cite{BSS} for a general discussion) that 
very weak, e.g., discontinuous, viscosity solutions 
of \eqref{motion-by-mean} and \eqref{nonlocal-motion}  may not be unique. 
The uniqueness is very much related to the issue of the development of 
interior. 

We say that an evolving front 
$(\Omega_t,\Gamma_t,\real^N\setminus\bar\Omega_t)$ 
does not develop interior if, for all $t\geqq 0$,
\begin{equation}\label{evolving-front}
\bigcup_{t>0} (\Gamma_t \times \{t\} ) 
= \partial \{ (x,t) : u(x,t) >0\} 
= \partial \{ (x,t) : u(x,t) <0\}\ .
\end{equation}

There are several sufficient conditions on 
$(\Omega_0,\Gamma_0,\real^N\setminus \bar\Omega_0)$ 
that imply that there is no interior (see \cite{BSS}). 
A general necessary and sufficient condition, which is related to the 
uniqueness of solutions of \eqref{pde}, 
is given in the next proposition.
For its proof we refer to \cite{BS2}. 

\begin{prop}\label{prop:BS}
For an open subset $\Omega_0$ of $\real^N$, let 
$(\Omega_t,\Gamma_t,\real^N\setminus \bar\Omega_t)$  
be the level-set evolution  of 
$(\Omega_0,\Gamma_0, \real^N\setminus\bar\Omega_0)$  
with normal velocity $-F$.

\noindent {\rm (i)} 
The no-interior condition \eqref{evolving-front}  holds 
if and only if it holds for $t=0$ and the initial value problem 
\eqref{pde} with initial datum  
$u_0= \bone_{\Omega_0} - \bone_{\real^N\setminus \bar\Omega_0}$ has a 
unique discontinuous viscosity solution.

\noindent {\rm (ii)} 
If \eqref{evolving-front} fails, then every usc subsolution 
(resp. lsc supersolution) $w$ of \eqref{pde} with 
$w(\cdot,0) \leqq \bone_{\bar\Omega_0}     - \bone_{\bar\Omega_0^c}$ 
(resp.  $w(\cdot,0) \geqq \bone_{\Omega_0} - \bone_{\bar\Omega_0^c}$) 
satisfies, in $\real^N\times (0,+\infty)$,
\begin{equation*}
w\leqq \bone_{\Omega_t^+\cup \Gamma_t} - \bone_{\Omega_t^-} \qquad 
\text{(resp. $w \geqq\bone_{\Omega_t^+} - \bone_{\Gamma_t \cup\Omega_t^-}$) .}
\end{equation*}
\end{prop}
\vskip18pt

%
As far as the kernels $p_\alpha$ are concerned we recall 
that, for $(x,t)\in\real^N \times (0,\infty)$, 
\begin{equation}\label{eq0.5}
p_\alpha (x,t) = t^{-N/\alpha} P_\alpha (xt^{-1/\alpha})\ ,
\end{equation}
where, for some $C_{N,\alpha} >0$ and all $x\in \real^N$, 
the kernel $P_\alpha (x) = p_\alpha (x,1)$ satisfies 
\begin{equation}\label{eq0.6}
0\leqq P_\alpha (x) \leqq C_{N,\alpha} (1+|x|^{N+\alpha})^{-1}  
\ \text{ and }\ 
|DP_\alpha (x)| \leqq C_{N,\alpha} |x|^{N-1+\alpha} 
(1+ |x|^{N+\alpha})^{-2}\ . 
\end{equation}

We will also use here that, both locally uniformly in $\real^N\setminus \{0\}$ 
and in $L^1(\real^N)$,  
\begin{equation}\label{eq0.7} 
\lim_{t\to0} t^{-1} p_\alpha (\cdot, t) = \tilde p_\alpha (\cdot)\ ,
\end{equation}
where, for some $C_{N,\alpha} >0$,  
\begin{equation}\label{eq0.8} 
\tilde p_\alpha (x) = C_{N,\alpha} |x|^{-(N+\alpha)}\ .
\end{equation} 

\section{The proof of the convergence}

The main step of the proof of the Theorem  is 

\begin{prop}\label{prop:main}
The functions $\limsup^* u_h$ and $\liminf_* u_h$ are, respectively, 
viscosity subsolutions and supersolutions of \eqref{pde} for the 
$F$ specified in the Theorem.
\end{prop}

We postpone the proof of Proposition~\ref{prop:main} and we proceed with the 

\begin{proof}[Proof of the Theorem] 
Let $u$ be as in the statement 
and denote by $\sign^*$ and $\sign_*$ the usc and the lsc 
envelopes respectively of the sign function in $\real$. 

The functions (see \cite{BSS} for the proof) 
$$\sign^* (u) = \begin{cases} 1&\text{if\quad $u\geqq 0$,}\\
\noalign{\vskip6pt}
-1&\text{if\quad $u<0$,}
\end{cases}
\qquad\text{and}\qquad
\sign_* (u) = \begin{cases} 1&\text{if\quad $u>0$,}\\
\noalign{\vskip6pt}
-1&\text{if\quad $u\leqq 0$,}
\end{cases}$$
are respectively the maximal usc subsolution and the minimal lsc 
supersolution of \eqref{pde} with initial datum 
$\bone_{\Omega_0} - \bone_{\bar\Omega_0^c}$ --- 
recall that $\Omega_0 = \{x:g(x)>0\}$.
Therefore any subsolution (resp. supersolution) 
$w$ of \eqref{pde} with the same initial datum  satisfies  
$$w\leqq \sign^* u\qquad \text{(resp. } w\geqq \sign_* u)
\quad \text{in }\ \real^N\times (0,\infty) \ .$$

It then follows from 
Proposition~\ref{prop:main} that	
\begin{equation}\label{eq:prop-main}
\limsup\nolimits^* u_h \leqq \sign^* u\quad\text{and}\quad 
\liminf\nolimits_* u_h \geqq \sign_* u\  \text{ in }\ 
\real^N\times (0,\infty)\ .
\end{equation}

Since $u_h$ takes only the values $\pm1$, \eqref{eq:prop-main} gives 
$$\liminf\nolimits_* u_h = 1\quad\text{in}\quad \Omega_t = \{u>0\}
\quad\text{ and }\quad 
\limsup\nolimits^* u_h = -1\quad \text{in}\quad \{u<0\}\ .$$

The proof of the first part of the Theorem is now complete.
\end{proof}

The Corollary follows exactly as in Section~5 of \cite{BG}.

We continue with the:  

\begin{proof}[Proof of Proposition~\ref{prop:main}] 
We only present the argument for $\bar u= \limsup^* u_h$. 
The claim for $\liminf_* u_h$ follows similarly.

Let $\phi$ be a smooth test function and assume that $(x_0,t_0) \in 
\real^N\times (0,\infty)$ is a strict global maximum point of $\bar u-\phi$.
To avoid any technical difficulties,  we 
assume that $\liminf_{(y,s)\to\infty} \phi (y,s) = +\infty$.  

If either $\bar u (x_0, t_0) =-1$  or 
$(x_0,t_0)$ belongs to the interior of the set $\{\bar u=1\}$,
the facts that $\bar u$ is  usc and takes only 
the values $\pm1$ yield that $\bar u=-1$ in a neighborhood of $(x_0,t_0)$.
Therefore,  
\begin{equation}\label{sub}
D\phi (x_0,t_0) =0\ ,\quad D^2 \phi (x_0,t_0) \geqq 0
\quad\text{ and }\quad \phi_t (x_0,t_0) =0\ .
\end{equation}

Next we assume that $(x_0,t_0) \in \partial \{\bar u=1\}$.
It is  standard  in the theory of viscosity solutions that,
since $\phi$ grows at infinity and $(x_0,t_0)$ is a strict maximum, 
there exists a subsequence $(x_h,n_hh)$ such that, as $h\to0$,  
\begin{equation}\label{max}
\begin{cases}
u_h^* (x_h,n_h h) - \phi (x_h,n_h h) = 
\max_{\real^N\times\nat}  (u_h^* -\phi) \ ,\\
\noalign{\vskip6pt}
u_h^* (x_h,n_hh)\to 1\quad\text{and}\quad 
(x_h,n_h h) \to ( x_0,t_0)\ ,
\end{cases} 
\end{equation}
where $u_h^*$ is the upper semicontinuous envelop of $u_h$.

Here we used the upper semicontinuity of $u_h^*$ 
and the fact that 
$$\limsup\nolimits^* u_h = \limsup\nolimits^* u_h^*\ .$$

Once again $u_h$ taking only the values $\pm 1$ and 
$u_h^* (x_h,n_hh)\to1$, as $h\to0$, imply that, for $h$ sufficiently small, 
$u_h^* (x_h,n_h h) =1$. 
Moreover, for all $x\in \real^N$ and $n\in\nat$,
\begin{equation}\label{eq2.2} 
u_h^* (x,nh) \leqq 1-\phi (x_h,nh) + \phi (x,nh)\ . 
\end{equation}

Indeed, if $u_h^* (x,nh) =-1$, the inequality is trivially true, while 
if $u_h^* (x,nh) =1$, then 
$$\phi (x,nh) - \phi (x_h,n_hh) \geqq 0\ ,$$
and, therefore, 
\begin{equation*}
u_h^* (x,nh) \leqq \sign^* (\phi (x,nh) - \phi (x_h,nh))\ .
\end{equation*}

Recall next that 
$$u_h (\cdot,n_h h) = \sign^* (S(h) u_h(\cdot,(n_h-1)) (\cdot))\quad
\text{in}\quad\real^N\ ,$$
where 
$$S(h) v= J_h * v\ .$$

Therefore, 	
$$u_h (\cdot ,n_h h) \leqq \sign^* (S(h) u_h^* (\cdot, (n_h -1)h)) 
\ \text{ in }\ \real^N\ ,$$
and, hence, 
$$u_h^* (\cdot ,n_h h) \leqq \sign^* (S(h) u_h^* (\cdot, (n_h -1)h)) 
\ \text{ in }\ \real^N\ .$$

Let $x=x_h$.
Since
$$1= u_h^* (x_h,n_h h) = \sign^* S(h) (u_h^* (\cdot,(n_h-1)h) (x_h)\ ,$$
we have 
$$0\leqq S(h)( u_h^* (\cdot, (n_h -1)h)(x_h) 
= S(h) \sign^* u_h^* (\cdot, (n_h -1)h)(x_h).$$

The definition of $S(h)$ and $u_h^*$ taking only  the values $\pm1$ yield 
\begin{equation}\label{eq2.3}
\begin{split}
0\leqq & \int 
(\bone^+ ( u_h^*(y+x_h, (n_h-1) h) - u_h^* (x_h,n_h h))\\ 
\noalign{\vskip6pt} 
&\qquad -\bone^- (u_h^*(y+x_h, (n_h-1) h) - u_h^* (x_h,n_h h))) 
p_\alpha (y,\sigma_\alpha (h))\,dy\ . 
\end{split}
\end{equation}


The proof will then be complete if we show that \eqref{eq2.3} implies, that,  
at $(x_0,t_0)$,  
if $D\phi  =0$, then $\phi_t \leqq 0$, or, if $|D\phi| \ne 0$, 
$\phi_t \leqq F(D^2 \phi,D\phi, \{\bar u (\cdot,t)\geqq \bar u(x_0,t_0)\})$,
with $F$ given by \eqref{motion-by-mean} if $\alpha \in [1,2)$ and 
\eqref{nonlocal-motion} if $\alpha \in (0,1)$. 

This is exactly the consistency of the scheme, which we investigate 
in the next section. 

\end{proof}
\vskip18pt

\section{The consistency}

Since the argument is technical,  
it is necessary to look at several different 
cases depending on the range of $\alpha$ and 
whether $D\phi $ vanishes or not. 
As discussed earlier we only check the subsolution property.

We begin with the case $\alpha \in (0,1)$. 
To this end let 
\begin{equation}\label{const1}
C_\alpha = \bigg( 2 \int_{\real^{N-1}} P_\alpha (0,y')\,dy'\bigg)^{-1}
C_{N,\alpha}\ ,
\end{equation}
where $C_{N,\alpha}$ is the constant in \eqref{eq0.8}. 

\begin{prop}\label{prop3.1}
Fix $\alpha \in(0,1)$, set $\sigma_\alpha (h) = h^{\alpha/1+\alpha}$ and assume 
that \eqref{eq2.3} holds. 
Then, for any ball $B_\delta \subset\real^N$ centered at $(x_0,t_0)$, 
\eqref{usc3} holds at $(x_0,t_0)$ with $C_\alpha$ given by \eqref{const1}.
\end{prop}

\begin{proof}
Throughout the proof, to make the notation simpler, we drop the explicit 
dependence of $p_\alpha$ and $\sigma_\alpha$ on $\alpha$, 
we write $\sigma$ instead of $\sigma (h)$,
$t_h = nh$ and $\phi_h (y,s) =\phi (y+x_h,s) - \phi (x_h,t_h)$.

In view  \eqref{max}, for any $A\subset \real^N$, we have 
\begin{equation}\label{eq2.5}
\left\{\begin{array}{l}
\{ y\in \real^N :u_h^* (y+x_h,t_h-h) \geqq u_h^* (x_h,t_h)\} \cap A 
\subseteq \{y\in \real^N : \phi_h (y,t_h-h) \geqq 0\} \cap A\ ,\\
\noalign{\vskip6pt}
\text{\qquad and}\\
\noalign{\vskip6pt}
\{ y\in \real^N :u_h^* (y+x_h,t_h-h) < u_h^* (x_h,t_h)\} \cap A 
\supseteq \{ y\in \real^N : \phi_h (y,t_h-h) <0\} \cap A\ .
\end{array}\right.
\end{equation}

Let $a= \phi_t (x_0,t_0)$. 
Since, as $h\to0$, $\phi_t(x_h,t_h)\to \phi_t (x_0,t_0)$, for any 
$\gamma>0$ and sufficiently small $h$, which may depend on $\delta$, we have 
\begin{equation}\label{small-h} 
\phi (\cdot + x_h,t_h-h) \leqq \phi (\cdot + x_h,t_h) - (a-\gamma) h
\ \text{ in }\ \tilde B_\delta \ ,
\end{equation}
and, hence, 
\begin{equation*} 
\phi_h (\cdot,t_h-h) \leqq \phi_h (\cdot,t_h) - (a-\gamma) h\ \text{ in }\ 
\tilde B_\delta\ ,
\end{equation*}
where now the ball $\tilde B_\delta$ is centered at the origin.

We consider next two different cases depending on whether 
$|D\phi (x_0,t_0)|$ vanishes or not.

If $|D\phi (x_0,t_0)| \ne0$, then choosing $\delta$ sufficiently we may 
assume that the level sets of $\phi$ in $B_\delta$ near $\phi (x_0,t_0)$ 
are graphs of smooth functions with uniformly bounded derivatives. 

If 
\begin{equation*}
\I_h = \int (\bone^+ (u_h^* (y+x_h,t_h-h) - u_h^* (x_h,t_h)) 
- \bone^- (u_h^* (y+x_h,t_h-h) - u_h^* (x_h,t_h)) 
p (y,\sigma)\,dy\ ,
\end{equation*}
in view of \eqref{eq2.5}, we have 
\begin{equation}\label{eq2.51}
\I_h\leqq \I_h^\delta + \II_h^\delta\ ,
\end{equation}
where 
\begin{equation*}
\I_h^\delta = \int_{\tilde B_\delta^c} 
(\bone^+ (u_h^* (y+x_h,t_h-h) - u_h^* (x_h,t_h)) 
- \bone^- (u_h^* (y+x_h,t_h-h) - u_h^* (x_h,t_h)) 
p (y,\sigma)\,dy
\end{equation*}
and 
\begin{equation*}
\II_h^\delta = \int_{\tilde B_\delta} (\bone^+ (\phi_h (y,t_h-h)) 
- \bone^- (\phi_h (y,t_h-h)) p(y,\sigma)\,dy\ .
\end{equation*}

Since, as $h\to0$, $u_h^* (x_h,t_h)\to u^* (x_0,t_0)$, it follows 
from \eqref{eq0.5}, \eqref{eq0.7} and the upper semicontinuity of $u_h^*$ 
that 
\begin{equation}\label{eq2.6}
\varlimsup_{h\to0} \sigma^{-1}  \I_h^\delta \leqq C_{N,\alpha}
\bar \I_{B_\delta^c} [u] 
(x_0,t_0)\ .
\end{equation}

%

Next we concentrate on the limiting behavior, as $h\to0$, 
of $\II_h^\delta$ for $\delta$ sufficiently small.

Using \eqref{small-h} we find that 
\begin{equation*}
\II_h^\delta \leqq \III_h^\delta + \IV_h^{\delta,+} + \IV_h^{\delta,-}
\end{equation*}
where 
\begin{equation*}
\III_h^\delta = \int_{\tilde B_\delta} (\bone^+ (\phi_h (y,t_h)) 
- \bone^- (\phi_h (y,t_h))) p (y,\sigma) \,dy \ ,
\end{equation*} 
and
\begin{equation*}
\IV_h^{\delta,\pm} = \pm \int_{\tilde B_\delta} 
(\bone^\pm (\phi_h (y,t_h) - (a-\gamma) h) 
- \bone^\pm (\phi_h (y,t_h) )) p (y,\sigma)\,dy\ .
\end{equation*}

We treat each one of these terms separately. 
As far as $\III_h^\delta$ is concerned, we observe that the Dominated 
Convergence Theorem and \eqref{eq0.7}  yield 
\begin{equation}\label{eq:DCT}	
\begin{split}
\varlimsup_{h\to0} \sigma^{-1} \III_h^\delta 
&\le C_{N,\alpha} \int_{B_\delta} 
 (\bone^+ (\phi (y+ x_0,t_0) -\phi (x_0,t_0))\\ 
&\qquad \qquad  - \bone^- (\phi (y+ x_0,t_0) -\phi (x_0,t_0))) 
|y|^{-(N+\alpha)} \,dy\\
\noalign{\vskip6pt}
& = C_{N,\alpha} \bar \I_{B_\delta}^\delta [\phi] (x_0,t_0)\ ,
\end{split}
\end{equation}
provided we can show that the integrand in $\III_h^\delta$ 
divided by $\sigma^{-1}$ is integrable. 

To this end, using radial coordinates and \eqref{eq0.6}, we find 
\begin{equation*}
|\sigma^{-1} \III_h^\delta| 
\leqq C_{N,\alpha} \int_0^\delta (\sigma^2 + r^2)^{-\frac{N+\alpha}2} 
D_h^\delta (r)\,dr
\end{equation*}
where 
\begin{equation*}
D_h^\delta (r) = \Big| \int_{\partial \tilde B(r)} 
(\bone^+ (\phi (y+x_h,t_h) -\phi (x_h,t_h)) 
- \bone^- (\phi (y+x_h,t_h) - \phi (x_h,t_h)))\,d\sigma\ .
\end{equation*}

In general we know that $D_h^\delta (r)$ is bounded for all $r$. 
However, it is a calculus exercise to check that the regularity of $\phi$ 
and the fact that $|D\phi| \ne0$ in $B_\delta$ --- recall that $\delta$ 
is sufficiently small and $|D\phi (x_0,t_0)| \ne0$, yield, for some 
constant $C>0$ depending on $\delta$, 
$$D_h^\delta (r) \leqq Cr\ .$$

Hence, since $\alpha \in (0,1)$, 
$$|\sigma^{-1} \III_h^\delta| \leqq \int_0^\delta r 
(\sigma^2 +r^2)^{-\frac{N+\alpha}2}\,dr <\infty\ .$$

We continue now with the analysis of $\pm \IV_h^{\delta,\pm}$.
We show that, for $\sigma = h^{\alpha/1+\alpha}$, 
\begin{equation}\label{eq2.81}
\varlimsup_{h\to0} \pm\sigma^{-1} \IV_h^{\delta,\pm} 
\leqq - (a-\gamma) |D\phi (x_0,t_0)|^{-1} 
(2\int_{\real^{N-1}} P(0,y')\,dy')\ . 
\end{equation}

Combining \eqref{small-h} and \eqref{eq2.81} yields, after letting 
$\gamma\to0$, \eqref{usc3} at least when $|D\phi (x_0,t_0)| \ne 0$.

We show \eqref{eq2.81} when $a>0$. 
The case $a\leqq 0$ follows similarly.
To this end, given that $|D\phi (x_0,t_0)| \ne0$, by taking $\delta$ 
even smaller if 
necessary, we may assume that $\phi_h (\cdot,t_h)$ has the form 
\begin{equation}\label{eq2.88}
\phi_h (y,t_h) = \beta_y y + (A_h y,y)\ ,
\end{equation}
with $A_h = 2 D^2 \phi (x_h,t_h)$ 
and $\beta_h =|D\phi_h (x_h,t_h)| \to |D\phi (x_0,t_0)|$,
as $h\to 0$.

We then have  
\begin{equation*}
\{ y\in\real^N : 0 \leqq \phi_h (\cdot,t_h) \} 
= \{ y\in \real^N : 0\leqq y_1 + (\tilde A_hy ,y)\}
\end{equation*}
and 
\begin{equation*}
\{ y\in \real^N : (a-\gamma) h \leqq \phi_h (\cdot,t_h)\}
= \{ y\in \real^N : a_h h \leqq y_1 + (\tilde A_h y,y)\}\ ,
\end{equation*}
where 
\begin{equation*}
\tilde A_h =\beta_h^{-1} A_h \quad\text{ and }\quad 
a_h = \beta_h^{-1} (a-\gamma)\ .
\end{equation*}

Consider the integrals 
$$J_h^{\delta,\pm} = \pm \int_{B_\delta^c} 
\bone^\pm (\phi_h (y,t_h) - (a-\gamma)h) - \bone^\pm (\phi_h (y,t_h)) 
p(y,\sigma)\,dy\ .$$

In view of \eqref{eq0.5} and \eqref{eq0.6} we have, 
for some $C_{N,\alpha}^\delta >0$ 
$$0\leqq p (y,\sigma) \leqq \sigma C_{N,\alpha}^\delta |y|^{-(N+\alpha)} 
\ \text{ in }\ \tilde B_\delta\ .$$

Moreover, as $h\to0$ and almost everywhere in $y$,  
$$| \bone^\pm (\phi_h (y,t_h) - (a-\gamma) h) -
\bone^\pm (\phi_h (y,t_h)|\to 0\ .$$

Therefore 
$$\lim_{h\to0} \sigma^{-1} J_h^{\delta,\pm} = 0\ .$$

Adding and subtracting $J_h^{\delta,\pm}$ to \eqref{eq2.51} we see that to 
conclude the proof we need to show that 
\begin{equation}\label{eq2.82} 
\begin{split}
&\varlimsup_{h\to0}{}^{\pm}  \sigma^{-1} \int_{\real^N} 
(\bone^\pm (\phi_h (y,t_h) - (a-\gamma) h) - \bone^\pm (\phi_h(y,t_h)))
p(y,\sigma) dy \\
&\qquad \leqq -		
(a-\gamma) |D\phi (x_0,t_0)|^{-1} 
\int_{\real^{N-1}} P(0,y') dy'\ .
\end{split}
\end{equation}
%

Let 
\begin{equation*}
\begin{split}
\psi (\sigma) & = \int_{\real^N} (\bone^+ (\phi_h (y,t_h) - (a-\gamma) h) 
- \bone^+ (\phi_h (y,t_h)) p (y,\sigma)\,dy\\
& =\int_{\real^N} (\bone^+ (\phi_h (\sigma^{1/\alpha} y,t_h) - (a-\gamma) h) 
- \bone^+ (\phi_h (\sigma^{1/\alpha} y,t_h))) P(y)\,dy
\end{split}
\end{equation*}
and 
$$\Psi_h(y) = (\tilde A_h y,y)\ .$$

Since $\sigma = h\sigma^{-1/\alpha}$ and 
$$\phi_h (\sigma^{1/\alpha} y,t_h) = \sigma^{1/\alpha} \beta_h y_1 
+ \sigma^{2/\alpha} (A_h y,y)\ ,$$
we have
\begin{equation*}
\bone^+ (\phi_h (\sigma^{1/\alpha} y,t_h) - (a-\gamma)h) 
= \bone^+ (y_1 +\sigma^{1/\alpha} (\tilde A_h y,y) - a_h h \sigma^{-1/\alpha})
= \bone^+ (y_1 +\sigma^{1/\alpha} \Psi_h(y) - a_h \sigma) 
\end{equation*}
and 
\begin{equation*} 
\bone^+ (\phi_h (\sigma^{1/\alpha} y,t_h)) = \bone^+ (y_1 +\sigma^{1/\alpha} 
\Psi_h (y))\ .
\end{equation*}

It is therefore immediate that $\psi(0)=0$. 
Hence to prove \eqref{eq2.82} we need to find $\psi'(0)$.

In view of the above simplifications we rewrite $\psi$ as 
\begin{equation*}
\psi (\sigma) = \int_{\real^N} (\bone^+ (y_1 + \sigma^{1/\alpha} \Psi_h (y) 
- a_h \sigma) - \bone^+ (y_1 + \sigma^{1/\alpha}  \Psi_h (y))) P(y)\,dy\ .
\end{equation*}

We state below 
the main step of the proof of \eqref{eq2.82} as a separate lemma. 
To this end, let 
\newline $f:[0,\infty)\to\real$ be defined by 
\begin{equation}\label{f-def}
f(\sigma) = \int_{\real^N} \bone^+ (y_1 + F(y,\sigma)) P(y)\,dy \ .
\end{equation}

We have:  

\begin{lem}\label{lem3.8}
Let $f$ be given by \eqref{f-def} and assume that 
$F\in C^1 (\real^N \times [0,\infty))$ and 
\newline
$\partial_{y_1s}^2 F\in C (\real^N\times [0,\infty))$. 
Then 		
\begin{equation}\label{eq2.92}
\begin{split}
f(\sigma) - f(0) 
& = - \int_0^\sigma \int_{\real^N} 
\bone^+ (y_1 + F(y,\sigma)) \partial_{y_1}(\partial_\rho F(y,\rho) P(y))\,dy \\
\noalign{\vskip6pt}
&\qquad 
-\int_{\real^N}\bone^+(y_1 + F(y,\sigma)) \partial_{y_1} F(y,\sigma) P(y)\,dy\\
\noalign{\vskip6pt}
&\qquad + \int_0^\sigma \int \bone^+ (y_1 + F(y,\rho))\partial_\rho 
(\partial_{y_1} F(y,\rho) P (y))\,dy\ .
\end{split}
\end{equation}
\end{lem}

We continue the ongoing proof and return to the proof of the lemma later.
We use Lemma~\ref{lem3.8} with 
\begin{equation*}
F(y,\sigma) = \sigma^{1/\alpha} (\tilde A_h y,y) - a_h\sigma
\quad\text{and}\quad 
F(y,\sigma) = \sigma^{1/\alpha} (\tilde A_h y,y)\ ,
\end{equation*}
both of which satisfy the assumption of Lemma~\ref{lem3.8}.  
In either case we find 
\begin{equation*}
\partial_{y_1}  F(y,\sigma) = 2\sigma^{1/\alpha} (\tilde A_{h1,1} y_1 
+ \sum_{j=2}^N \tilde A_{h,j} y_j)\ .
\end{equation*}

Therefore, 
\begin{equation*}
\lim_{\sigma\to0} \sigma^{-1} [ \int_0^\sigma \int 
\bone^+ (y_1 + F(y,\rho))\partial_\rho (\partial_y , F(y,\rho)) P(y) 
- \int \bone^+ (y_1 + F(y,\sigma)) \partial_y , F(y,\sigma) P(y)\,dy ]
= 0\ .
\end{equation*}

Since $\alpha \in (0,1)$, it is hence immediate that 
\begin{equation*}
\psi' (0) = a_h \int_{\real^N} \bone^+(y_1) \partial_yP(y)\,dy 
= - a_h \int_{\real^N} {\boldsymbol\delta} (y_1) P (y) \,dy 
= - \frac{a-\gamma}{|D\phi (x_h,t_h)|} 
\int_{\real^{N-1}} P(y')\,dy\ ,
\end{equation*}
where $\boldsymbol\delta$ denotes the usual Dirac mass. 

Next assume that $|D \varphi (x_0,t_0)| =0$. 
Since, as $h\to0$, 
 $\beta_h = |D\phi (x_h,t_h)|\to0$, the previous argument does not work
and it is necessary to look at several cases. 
It is clear from the definition of the solution that we only have to show 
that 
$$a = \phi_t (x_0,t_0) \leqq 0\ .$$

First we assume that, along some sequence $h\to0$, $\beta_h\ne 0$ and 
$\sigma \beta_h^{-1}\to0$.  
In this case it is possible to repeat the previous argument. 
The main difference is that, instead of a fixed $\delta$, 
here we use 
$\delta_h = \beta_h C$ with $C= 2\|D^2\phi\|$. 

We look at each of the estimates/limits of the first part. 
We begin with $\I_h^\delta$.
Using that 
$$|\I_h^{\delta_h}| \leqq C\,\delta^{-\alpha}$$
we find 
$$\beta_h |\I_h^\delta| \leqq C\,\beta_h \delta_h^{-\alpha} 
= C\, \beta_h^{1-\alpha}\ .$$

For the limit involving $\III_h^\delta$ we observe that 
the integration takes place over a set $D_h$ such that
\begin{equation*}
\begin{split}
D_h &\subset \{y\in B_{\delta_h} : \beta_h |y_1|\leqq C(|y_1|^2 + |y'|^2)\}
\subset \{y\in B_{\delta_h} : \beta_h |y_1| \leqq C\, \delta_h |y_1| 
+ C|y'|^2\}\\
\noalign{\vskip6pt}
&\subseteqq \{ y\in B_{\delta_h} : |y_1| \leqq \beta_h(2C)^{-1} 
|y'|^2\}\ ,
\end{split}
\end{equation*}

Similarly, it follows that
$$\limsup_{h\to0} \sigma^{-1} \beta_h
|J_h^{\delta_h,\pm}| =0\ .$$

The only thing left to check now is that 
\begin{equation*}
\begin{split}
&\varlimsup_{h\to0} (\pm)\sigma^{-1} \beta_h
\int 
(\bone^\pm ( y_1 + \sigma^{1/\alpha} \beta_h^{-1} (A_hy,y) 
- (a-\gamma) \sigma \beta_h^{-1} ) 
- \bone^\pm ( y_1 +\sigma^{1/\alpha} \beta_h^{-1} (A_hy,y)) p(y,\sigma)\,dy\\
&\qquad \leqq - (a-\gamma) \int_{\real^{N-1}} P(0,y')\,dy'\ .
\end{split}
\end{equation*}

This, however, follows exactly as before. 
Notice that, since $\alpha <1$,  
$\sigma^{1/\alpha} \beta_h^{-1} 
= \sigma^{\frac1{\alpha}-1} \sigma\beta_h^{-1} \to0$, as $h\to0$. 

The next case is that, along a subsequence $h\to0$, either 
$\beta_h=0$ or $\sigma\beta_h^{-1}\to\infty$. 
Here we argue by contradiction and assume that $a>0$, and, hence, 
$a-\gamma >0$ for $\gamma$ sufficiently small.

Arguing as in the beginning of the proof of  the $|D\phi (x_0,t_0)|\ne0$ case, 
we find 
\begin{equation}\label{eq2.99}
\left\{\begin{array}{l}
\ds 0 \leqq C\sigma \delta^{-\alpha} + \mkern-8mu\int_{B_\delta}\mkern-10mu
(\bone^+ (\beta_h y_1\! +\! (A_hy,y)\! -\! (a\!-\!\gamma)) 
- \bone^- (\beta_h y_1\! +\! (A_h y,y)\! -\!(a\!-\!\gamma))) p(y,\sigma)\,dy\\
\noalign{\vskip6pt}
\ds = C\sigma\delta^{-\alpha} + \int_{\tilde B_{\delta\sigma^{-1/\alpha}}} 
 (\bone^+ ( \sigma^{-1} \beta_h y_1 + \sigma^{\frac1{\alpha}-1} 
(A_hy,y) - (a-\gamma))) \\
\noalign{\vskip6pt}
\ds - \bone^- ( \sigma^{-1} \beta_h  y_1 + \sigma^{\frac1{\alpha}-1} 
(A_h y,y) - (a-\gamma)) ) P(y)\, dy\ .
\end{array}\right.
\end{equation}

It is clear that, as $h\to0$  and a.e. in $y_1$, 
$$\bone^+ ( \sigma^{-1} \beta_h  y_1 + \sigma^{\frac1{\alpha}-1} 
(A_h y,y) - (a-\gamma))\to0
\ \ \text{ and }\ \ 
\bone^- ( \sigma^{-1} \beta_h  y_1 + \sigma^{\frac1{\alpha}-1} 
(A_h y,y) - (a-\gamma))\to1\ .$$

Hence dividing \eqref{eq2.99} by $\sigma$ and letting $h\to0$ we get 
a contradiction to $a>0$. 

The last case to consider is that, 
along a sequence $h\to0$, $\beta_h\ne0$ and $\sigma \beta_h^{-1} \to\ell$
with $\ell>0$.

Again we rewrite \eqref{eq2.99} as 
\begin{equation}\label{eq2.999}
\begin{split}
0\leqq C\sigma\delta^{-\alpha} + \int_{\real^N} 
&(\bone^+ ( y_1 + \sigma^{1/\alpha} \beta_h^{-1} (A_hy,y) - (a-\gamma)
\sigma \beta_h^{-1}) \\
\noalign{\vskip6pt}
&- \bone^- ( y_1 +\sigma^{1/\alpha} \beta_h^{-1} (A_h y,y) -(a-\gamma) 
\sigma \beta_h^{-1}) ) P(y)\,dy\ .
\end{split}
\end{equation}

Since $\sigma^{1/\alpha} \beta_h^{-1} = \sigma^{\frac1{\alpha}-1}
\sigma\beta_h^{-1} \to0$, as $h\to0$, letting $h\to0$ and $\gamma\to0$ 
in \eqref{eq2.999} gives 
\begin{equation*}
0\leqq \int_{\real^N} (\bone^+ ( y_1 - a\ell) 
- \bone^- ( y_1 - a \ell)) P(y)\,dy\ ,
\end{equation*}
which implies, in view of the symmetry of $P$, that we must have 
$a\leqq 0$.
\end{proof}

We continue now with the 

\begin{proof}[Proof of Lemma~\ref{lem3.8}]
To keep the ideas clear we present a formal proof using Dirac masses etc..
Everything can, of course, be made rigorous considering smooth approximations 
of $\bone^+$ and passing to the limit.
We leave it up to the reader to do so.

For $\rho>0$  we have 
\begin{equation*}
\begin{split} 
f'(\rho) & = \int_{\real^N} {\boldsymbol\delta}
 (y_1 + F(y,\rho)) \partial_\rho F(y,\rho)
P(y)\,dy\\
\noalign{\vskip6pt}
& = \int_{\real^N} \partial_{y_1} (\bone^+ (y_1 +F(y\rho)))
\partial_\rho F(y,\rho) P(y)\,dy
- \int_{\real^N} {\boldsymbol\delta}
 (y_1 +F(y,\rho)) \partial_\rho F(y,\rho) 
\partial_{y_1} F(y,\rho) P(y)\,dy\\ 
\noalign{\vskip6pt}
& = - \int_{\real^N} \bone^+ (y_1 +F(y,\rho)) \partial_{y_1} 
[\partial_\rho F(y,\rho) P(y) ]\,dy
- \int_{\real^N} (\partial_\rho \bone^+ (y_1 + F(y,\rho)))
\partial_{y_1} F(y,\rho) P(y)\,dy\ .
\end{split}
\end{equation*}

Therefore
\begin{equation*}
\begin{split}
f(\sigma) - f(0)
& = - \int_0^\sigma \int \bone^+ (y_1 +F(y,\rho)) \partial_{y_1} 
[\partial_y F(y,\rho) P(y)]\,dy\\
\noalign{\vskip6pt}
&\qquad - \int_{\real^N} \bone^+ (y_1 + F(y,\sigma))\partial_{y_1} 
F(y,\sigma) P(y)\,dy\\
\noalign{\vskip6pt}
&\qquad + \int_0^\sigma \int_{\real^N} \bone^+ (y_1+F(y,\rho)) 
\partial_\rho (\partial_{y_1} F(y,\rho) P(y))\,dy\ .
\end{split}
\end{equation*}

\end{proof}

We proceed now with the case $\alpha \in [1,2)$ where, with the 
appropriate choice of $\sigma_\alpha$, we obtain a front moving 
by a weighted mean curvature. 
We treat two separate cases namely $\alpha\in (1,2)$ and $\alpha=1$. 
The reason is that the former is more or less straightforward, while the 
latter requires a bit more delicate analysis due to the border-line
integrability properties of the kernel $p_\alpha$.

Let 
\begin{equation}\label{eq:alpha-in12c}
C_\alpha = \bigg[ 2\int_{\real^{N-1}}P_\alpha (0,y')\, dy'\bigg]^{-1}
\int y_2^2 P_\alpha (0,y')\,dy\ .
\end{equation}

We have: 

\begin{prop}\label{prop:alpha-in12} 
Assume $\alpha \in (1,2)$ and set $\sigma_\alpha (h) = h^{\alpha/2}$. 
If \eqref{eq2.3}  holds, then, at $(x_0,t_0)$, we have 
\eqref{eq:usc1} with $C_\alpha$ given by \eqref{eq:alpha-in12c} if 
$|D\phi| \ne0$ or $\phi_t \leqq 0$ if $D\phi=0$ and $D^2\phi=0$.
\end{prop}

\begin{proof} 
As in Proposition~\ref{prop3.1}, to simplify the notation we write $\sigma$ 
for $\sigma_\alpha (h)$, $t_h$ for $nh$, 
and $\phi_h (y,t)$ for $\phi(x_h+y,t)-\phi (x_h,t_h)$. 

The starting point is again \eqref{eq2.3} which, using 
\eqref{max} and after a rescaling, implies   
\begin{equation}\label{eq3.12bis}
0\leqq \int
(\bone^+ ( \phi_h (\sigma^{1/\alpha} y,t_h -h)) - 
\bone^- ( \phi_h (\sigma^{1/\alpha} y,t_h -h))) P(y)\, dy\ . 
\end{equation}

Expanding $\phi_h$, using that 
$\phi_h (0,t_h) =0$, $\phi_h \in C^{2,1}$ and 
$\sigma  =h^{\alpha/2}$, we find 
$$\phi_h (\sigma^{1/\alpha} y,t_h-h) 
= \sigma^{1/\alpha} (p_h,y) + \sigma^{2/\alpha} 
( (A_hy,y) - a_h + O(\sigma^{1/\alpha} |y| +\sigma^{2/\alpha}) 
(|y|^2 +\sigma^{2/\alpha}))\ ,$$
where
$$p_h = D\phi_h (x_h,t_h) \ ,\quad 
A_h = \frac12 D^2 \phi_h (x_h,t_h)\quad \text{ and }\quad 
a_h = \phi_{h t} (x_h,t_h)\ .$$

After a rotation  and a change of variables, we may assume 
that $p_h = \beta_h e_1$ with $\beta_h = |p_h|$. 
We denote by $\tilde A_h$ the matrix we obtain from $A_h$ after the 
rotation. 
The integration in \eqref{eq3.12bis} is 
taking place over the sets $C_h$ and $C_h^c$, where 
\begin{equation*}
C_h 
= \{y\in \real^N : \sigma^{1/\alpha} \beta_h y_1 + \sigma^{2/\alpha} 
((\tilde A_h y, y) - a_h 
+ O ((\sigma^{1/\alpha} |y|^2 + \sigma^{2/\alpha})
(|y|^2 + \sigma^{2/\alpha})) \geqq 0\}\ .
\end{equation*}

We argue now as in the proof of Proposition~3.1 and \cite{BG}. 
The difference with the former is that 
$\alpha\in (1,2)$ gives integrability at the origin.
The difference with the latter is that here the kernel only has algebraic 
decay while in \cite{BG} it is an exponential.

Here we only present the argument if $|D\phi (x_0,t_0)|\ne0$. 
If $|D\phi (x_0,t_0)| =0$ and $D^2\phi (x_0,t_0)=0$, 
it is necessary to look again at different cases 
as in the proof of Proposition~3.1 and \cite{BG}. 
The argument is considerably simpler than the one presented in 
Proposition~3.1 since we do not need to consider special balls, etc..
We leave the details to the reader. 

Since, as $h\to0$, $\beta_h\to |D\phi (x_0,t_0)|\ne0$, 
the $\beta_h$'s are 
strictly positive for sufficiently small $h$.

Therefore 
$$C_h = \{y\in \real^N : 
y_1 + \sigma^{1/\alpha} \beta_h ((\tilde A_h,y,y) 
- a_h + O (\sigma^{1/\alpha} |y| +\sigma^{2/\alpha}) 
(|y|^2 + \sigma^{2/\alpha})) \geqq 0\}\ .$$

Finally using that, as $h\to0$, $\tilde A_h \to \tilde A$ and $a_h\to a$, 
where $\tilde A$ is the rotated matrix $A$ 
and $a= \phi_t (x_0,t_0)$, we find 
$$C_h = \{ y\in \real^N : y_1 + \sigma^{1/\alpha} \beta_h 
( (\tilde Ay,y) - a+ O (\sigma^{1/\alpha} |y| + \sigma^{2/\alpha}) 
(|y|^2 +\sigma^{2/\alpha}) 
+ o(1) (|y|^2 +1) \geqq 0\}\ .$$

After all the above reductions we are left with the inequality  
\begin{equation}\label{eq144}		
0 \leqq \int (\bone^+ (\Psi_h(y)) - \bone^- (\Psi_h(y))) P(y)\,dy\ ,`
\end{equation}
where, for $y\in \real^N$,  
$$\Psi_h (y) = y_1 + \sigma^{1/\alpha} (\beta_h)^{-1} 
[  (\tilde Ay,y) - a+ O (\sigma^{1/\alpha} |y| +\sigma^{2/\alpha})
(|y|^2 +\sigma^{2/\alpha}) 
+ o(1) (|y|^2 +1)]\ .$$

Let $R = \sigma^{-\frac{\theta}{\alpha}}$ for some $\theta > 0$. 
Then 
$$\int_{\real^N} \bone^\pm (\Psi_h (y)) P(y)\,dy 
= \int_{B_R} \bone^\pm (\Psi_h(y)) P(y)\,dy 
+ \int_{B_R^c} \bone^\pm (\Psi_h(y)) P(y)\,dy\ .
$$

Using \eqref{eq0.6}, we find, for some $c_{N,\alpha}>0$, that 
$$\int_{\real^N\setminus B_R} 
\bone^\pm (\Psi_h(y)) P(y)\,dy 
\leqq C_{N,\alpha} \int_{|y|\geqq R} 
(1+|y|^2)^{-\frac{N+\alpha}2}\,dy 
\leqq c_{N,\alpha} R^{-\alpha}\ .$$

Fix $\gamma>0$.  
For $h$ small  we have $\Psi_h \leqq \Psi^h$ in  $B_R$, where  
$$\Psi^h (z) = z_1 + \sigma^{1/\alpha} \beta_h^{-1} 
(\frac12 (\tilde A+\gamma \Id) - a\ .
$$     

Hence 
$$\int_{B_R} 
\bone_{\{ \Psi_h\geqq 0\}}(y) P (y)\, dy 
\leqq \int_{\real^N} \bone_{\{ \Psi^h\geqq 0\}}(y) P  (y)\, dy\ .$$

%
%
%
%
%
%

We summarize the above, using that
\begin{equation*}
\int \bone^+ (\Psi_h(y)) P(y)\,dy 
\leqq \int \bone^+ (\Psi^h(y)) P(y)\,dy\quad\text{and}\quad
\int \bone (\Psi^h (y)) P(y))\,dy 
\leqq \int\bone^- (\Psi_h (y)) P(y)\,dy\ ,
\end{equation*}
in the inequality 
\begin{equation}\label{eq145}
0\leqq \int (\bone^+ (y_1 + \rho \beta_h^{-1} \psi (y)) 
- \bone^- (y_1 + \rho\beta_h^{-1} \psi (y))) P(y)\,dy\ ,
\end{equation}
where $\rho = \sigma^{1/\alpha}$ and 
$$\psi (y) =  (\tilde A+\gamma \Id) -a\ .$$

Let 
$$f(\rho) = \int (\bone^+ (y_1 +\rho \beta_h^{-1} \psi (y)) 
- \bone^- (y_1 + \rho \beta_h^{-1} \psi (y))) P(y)\,dy\ .$$

The properties of $P$ yield that $f(0)=0$, therefore, as in the proof 
of Proposition~\ref{prop3.1}, we use Lemma~\ref{lem3.8} to find $f'(0)$.

It is a straightforward computation to see that it yields to the inequality 
$$a\leqq C_\alpha [\tr (\tilde A - \tilde A_{1,1}) - 2^{-1} (N+1)\gamma ]\ .$$

An elementary linear algebra calculation yields 
$$\tr (\tilde A - \tilde A_{1,1}) = \frac12 \tr 
(I - \widehat{D\phi (x_0,t_0)} \otimes \widehat{D\phi (x_0,t_0)} )
D^2\phi (x_0,t_0) \ ,$$
and, hence, after letting $\gamma\to0$, to the desired inequality.
\end{proof}

We continue with the case $\alpha=1$.
The argument is very similar to the one above. 
There is, however, a technical complication due to the logarithmic 
integrability of the kernel  $p_1$.
To deal with this difficulty, it is necessary to choose $\sigma_1$ 
differently. 

Let 
\begin{equation}\label{star-A}
C_1 = (2\int_{\real^{N-1}} P_1(0,y')dy')^{-1} 
\omega_{N-1}
\lim_{R\to\infty} (R^{N+1}  P_1 (0,R))\ ,
\end{equation}
where $\omega_{N-1}$ is the area of the unit sphere in $\real^{N-1}$ 
and, as always, $P_1$ is defined by \eqref{eq0.5}.

We have:

\begin{prop}\label{prop:alpha-equals-one}
Assume $\alpha=1$ and choose $\sigma$ so that 
$h = \sigma^2_1(h) |\ln \sigma_1(h)|$. 
If \eqref{eq2.3}  holds, then, at $(x_0,t_0)$,
\eqref{eq:usc1} holds with $C_1$ given by \eqref{star-A}  
if $|D\phi|\ne0$ or $\phi_t \leqq 0$, if $D\phi = 0$ and $D^2\phi=0$.  
%
\end{prop}

\begin{proof}
We only discuss the case $|D\phi (x_0,t_0)| \ne0$. 
The argument when $|D\phi (x_0,t_0|=0$ and $D^2\phi (x_0,t_0)=0$, is similar to 
the one in Propositions~\ref{prop3.1} and \ref{prop:alpha-in12}, 
hence we omit the details.
The proof follows very closely the one presented for the case $\alpha\in(1,2)$,
hence, again we only sketch the main steps.
The main difference/difficulty is the logarithmic integrability of the kernel.
To circumvent this potential problem it is necessary to make a
 ``more involved'' scaling.

To this end for each $\delta>0$, arguing as before, we reach the inequality
\begin{equation}\label{eq3.17}
0 \leqq \int_{|y|\leqq \delta\sigma^{-1}} 
[\bone^+ (y_1 + F(y,\rho)) - \bone^- (y_1 + F(y,\rho))] P(y)\,dy 
+ C_{N,1} \delta^{-1}\sigma
\end{equation}
where, for $\rho = \sigma |\ln\sigma|$,
\begin{equation*}
F(y,\rho) = 
\sigma ((\tilde A_h + \gamma_h \Id)y,y) -\rho a_h
\end{equation*}
with $\tilde A_h = \beta_h^{-1} \tilde A$, $\gamma_h = \beta_h^{-1}\gamma$
and $a_h = \beta_h^{-1}\gamma$.

Define 
\begin{equation*}
f(\rho) = \int_{|y|\leqq \delta\sigma^{-1}} 
(\bone^+ (y_1 +F (y,\rho) - \bone^- (y_1 + F(y,\rho ))) P(y)\,dy\ .
\end{equation*}

Using the decay properties of $P$ and the dominated convergence 
theorem we find easily that $f(0)=0$.

As before we need to calculate $f'(0)$.
Arguing formally---the calculation can be justified rigorously using 
regularizations of $\bone^+$ and $\bone^-$, etc.,---we get 
\begin{equation*}
\begin{split}
f'(\rho) & = 2\int_{|y|\leqq \delta\sigma^{-1}} 
{\boldsymbol\delta} (y_1 + F(y,\rho)) \partial_\rho (F(y,\rho)) P(y)\,dy\\
\noalign{\vskip6pt}
&\qquad + \int_{|y| = \delta\sigma^{-1}} 
(\bone^+(y_1 + F(y,\rho)) - \bone^- (y_1 + F(y,\rho))) P(y)\, 
d\Sigma^\sigma  (\delta\sigma^{-1})'
\end{split}
\end{equation*}
where $d\Sigma^\sigma $ is the surface measure 
on $\partial \tilde B_{\delta\sigma^{-1}}$ and $\prime$ denotes 
differentiation with respect to $\rho$. 

Then
\begin{equation*}
f(\rho) - f(0) = \I_\rho + \II_\rho\ ,
\end{equation*}
with 
\begin{equation*}
\I_\rho = 2\int_0^\rho \int_{|y|\leqq \delta \sigma^{-1}(\lambda)} 
{\boldsymbol\delta} (y_1 + F(y,\lambda))\partial_\lambda (F(y,\lambda)) P(y)\, dy
\, d\lambda\ ,
\end{equation*}
and 
\begin{equation*}
\begin{split}
\II_\rho &= \int_0^\rho \int_{|y|=1} 
[(\bone^+ (\delta\sigma^{-1} (\lambda) y_1 
+ F(\delta\sigma^{-1}(\lambda) y,\lambda)) \\
\noalign{\vskip6pt}
&\qquad 
- \bone^- (\delta\sigma^{-1}(\lambda) y_1 + F(\delta\sigma^{-1}(\lambda)y,
\lambda))] P(\delta\sigma^{-1}(\lambda) y) (\delta\sigma^{-1})^{N-1} 
d\Sigma  
\,d\lambda
\end{split}
\end{equation*}
where $d\Sigma $ is the surface measure on $\partial B_1$ 
and $\lambda = \sigma (\lambda) |\ln \sigma(\lambda)|$.

The growth of $P$ and the fact that  $\sigma'(\lambda)= (|\ln\sigma|-1)^{-1}$
yield, for some $C'>0$, the estimate
\begin{equation*}
\rho^{-1} |\II_\rho| \leqq \rho^{-1} C' \int_0^\rho 
\frac{(\delta\sigma^{-1})^{N-1}}{(1+(\delta\sigma^{-1})^2)^{\frac{N+1}2}}
(-\frac{\delta\sigma'}{\sigma^2})\,d\lambda
\leqq \frac{C}{\delta} \rho^{-1}\int_0^\rho 
(|\ln \sigma| -1)^{-1}
\end{equation*}
and, hence, 
$$\lim_{\rho\to0} \rho^{-1} |\II_\rho| = 0\ .$$

Next we analyze $I_\rho$. 
We begin with the observation that 
$$\partial_\lambda F(y,\lambda) = (\tilde A_h y,y) \sigma' -a\ .$$

Hence
$$\I_\rho = \I^1_\rho + \I^2_\rho\ ,$$
with 
$$\I^1_\rho= -2a\int_0^\rho \int_{|y|\leqq \sigma \lambda^{-1}} 
{\boldsymbol\delta} (y_1 + F(y,\lambda)) P(y)\,dy\, d\lambda$$
and 
$$\I_\rho^2 = \int_0^\rho (|\ln\sigma|-1)^{-1} \int_{|y|\leqq\delta\sigma^{-1}}
{\boldsymbol\delta} (y_1 + F(y,\lambda)) (\tilde A_h y,y) P(y)\,dy\,d\lambda\ .$$

It is immediate that 
$$\lim_{\rho\to0} \rho^{-1} \I_\rho^1 = - 2a \int_{\real^{N-1}} P(0,y')dy'\ ,$$
while 
$$\rho^{-1} \I_\rho^2 = \int_0^\rho 
\ln (\delta \sigma^{-1}) (|\ln \sigma|-1)^{-1} (\ln \delta \sigma^{-1})^{-1} 
\int_{|y|\leqq \delta\sigma^{-1}}
{\boldsymbol\delta} (y_1 + F(y,\lambda)) (\tilde A_h y,y) P(y)\,dy\, 
d\lambda\ ,$$
and, hence, 
$$\lim_{\rho\to0} \rho^{-1} \I_\rho^2 
= \frac1{(N-1)} \lim_{R\to\infty} 
(\frac1{\ln R} \int_{|y|\leqq R} |y'|^2 P(0,y')\,dy) 
\tr (\tilde A- \tilde A_{1,1})\ .$$

Returning now to \eqref{eq3.17} we find 
\begin{equation}\label{eq3.18}
0\leqq \rho^{-1} \int_{|y|\leqq \delta \sigma^{-1}}
(\bone^+ (y_1 +F (y,\rho)) -\bone^- (y_1 + F(y,\rho))) P(y)\,dy 
+ C_{N,1} \frac1{\delta|\ln\sigma|}\ .
\end{equation}

Letting $\rho\to0$ yields 
\begin{equation}\label{eq3.19}
\begin{split}
2a  \int_{\real^{N-1}} P(0,y')\,dy 
&\leqq \frac1{(N-1)} \lim_{R\to\infty} \frac1{\ln R} 
\int_{|y|\leqq R} |y'|^2 P(0,y')\,dy\\
& = \frac1{N-1}\ \omega_{N-1} \lim_{R\to\infty}(R^{N+1} P_1(0,R))\ .
\end{split}
\end{equation}

We may  now conclude as in Proposition~3.2.
\end{proof}

\end{document}